\documentclass[12pt, oneside]{article}

\usepackage[english]{babel}
\usepackage[T1]{fontenc}

\usepackage{amssymb}
\usepackage{amsmath}
\usepackage{hhline}
\usepackage{longtable}
\usepackage{amscd}
\usepackage{theorem}
\usepackage{array}
\usepackage{delarray}
\usepackage{multicol}
\usepackage{makeidx}

\usepackage{float}

\usepackage[pdftex]{graphicx}

\newtheorem{theorem}{Theorem}[section]

\newtheorem{lemma}{Lemma}[section]

\newenvironment{proof}[1][Proof]{\textbf{#1.} }{\ \rule{0.5em}{0.5em}}

\newtheorem{PB}{Probl\`{e}me}

\newtheorem{EX}{Exercise}

\newtheorem{EXC}{Exercice}

\newtheorem{so}{Exercice}

\newtheorem{sop}{Probl\`{e}me}

\begin{document}
  \begin{center}
\textbf{\LARGE The $b$-Chromatic Number and \\\vspace{5pt}
$f$-Chromatic Vertex Number  of \\\vspace{5pt} Regular Graphs}
\end{center}
\begin{center}
Amine El Sahili - Hamamache Kheddouci\end{center} \begin{center}
Mekkia Kouider - Maidoun Mortada\footnote{\noindent Corresponding address: Faculty of Sciences-Lebanese University, El Hadas-Beirut, Lebanon\\
 E-mail address: maydoun.mortada@liu.edu.lb\\
Tel: 009613065271- Fax: 009615461496}
\end{center}

 \begin{abstract}
   The $b$-chromatic number of a graph $G$, denoted
 by $b(G)$, is the
largest positive integer $k$ such that there exists a proper
coloring for G with $k$ colors in which every color class contains
at least one vertex adjacent to some vertex in each of the other
color classes, such a vertex is called a dominant vertex.
 The $f$-chromatic vertex number of a $d$-regular graph $G$,
denoted by $f(G)$, is the maximum number of dominant vertices of
distinct colors in a proper coloring with $d+1$ colors. El Sahili
and Kouider  conjectured that $b(G)=d+1$ for any $d$-regular
graph $G$ of girth 5. We study this conjecture by
giving some partial answers under supplementary conditions.\end{abstract}
\section{Introduction}
  For a graph $G=(V,E)$, $V(G)$ and $E(G)$ will denote its
vertex set and edge set, respectively. We denote by $N_G(x)$ the set
of the neighbors of the vertex $x$ in G and by $N^G_2(x)$ the set of
its second neighbors.
 The degree of the vertex $x$ in $G$ is the cardinality of the
set $N_G(x)$ and it is denoted by $d_G(x)$. For short, we use $d(x)$
instead of $d_G(x)$ and $N(x)$ instead of $N_G(x)$. The largest
degree in $G$ is denoted by $\Delta(G)$. A graph $G$ is said to be
$k$-regular if $d(v)= k$ for all $ v \in G$. A proper coloring of a
graph $G$ is a mapping $c:V\rightarrow S$ such that $c(u)\neq c(v)$
whenever $u$ and $v$ are adjacent. The set $S$ is the set of
available colors. A proper coloring with $m$ colors is usually
called an $m$-coloring. The chromatic number of $G$, denoted by $\chi
(G)$, is the smallest integer $m$ such that $G$ has an $m$-coloring.
A color class in a proper coloring of a graph $G$ is the subset of
$V$ which contains all the vertices with same color. A vertex of
color $i$ is said to be a dominating or dominant vertex if it has a
neighbor in each color class distinct from $i$. A color $i$ is said
to be a dominant color in $G$ if there exists a dominant vertex of color
$i$. A proper coloring of a graph is called a $b$-coloring, if each
color class contains at least one dominant vertex. The $b$-chromatic
number of a graph $G$, denoted by $b(G)$, is the largest positive
integer $k$
such that $G$ has a $b$-coloring with $k$ colors.\\
 The concept of the $b$-chromatic has been  introduced by Irving and
Manlove [23] 
 when considering minimal proper colorings with
respect to a partial order defined on the set of all partitions of
the vertices of a graph. They proved that determining  $b(G)$
 is NP-hard for general graphs, but polynomial-time solvable for
trees.\\
 Recently, Kratochvil \emph{et al.} [11]   have shown that
determining $b(G)$ is NP-hard even for bipartite graphs while
Corteel, Valencia-Pabon, and Vera [25] proved that there is no
constant $\epsilon > 0$ for which the $b$-chromatic number can be
approximated within a factor of 120/133-$\epsilon$ in polynomial
time (unless P=NP).\\
 Obviously, each coloring of $G$ with
$\chi(G)$ colors is a $b$-coloring. Also, $b(G)\leq \Delta(G)+1$.
Therefore, for each $d$-regular graph $G$, $b(G)\leq d+1$. Since
$d+1$ is the maximum possible $b$-chromatic number of $d$-regular
graphs, determining necessary or sufficient conditions to achieve
this bound is of interest.
 Hoang and Kouider [7] characterized all bipartite
graphs $G$ and all $P_4$-sparse graphs $G$ such that each induced
subgraph $H$ of $G$ satisfies $b(H) = \chi(H)$. If we are limited to
regular graphs, Kratochvil \emph{et al.} proved in [11] that for a
$d$-regular graph $G$ with at least $d^4$ vertices, $b(G) = d+1$. In
[24], Cabello and Jakovac reduced the previous bound to $2d^3-d^2+d$
and then El Sahili \emph{et al.} [2] showed that $b(G)=d+1$ for a
$d$-regular graph with at least $2d^3 +2d-2d^2$. It was also proved
in [2] that $b(G) = d+1$ for a $d$-regular graph $G$ containing no
cycle of order 4 and  with  at least $d^3 + d$ vertices.  It follows
from the above results that for any $d$, there is only a finite
number of $d$-regular graphs $G$ with $b(G) \leq d$. Kouider [15]
proved that the $b$-chromatic number of a $d$-regular graph of girth
at least 6 is $d+1$. El Sahili and Kouider [1] proved that the
$b$-chromatic number  of any $d$-regular graph of girth 5 that
contains no cycle of order 6 is $d+1$. In [1], El Sahili and
Kouider asked whether it is true that every $d$-regular graph $G$
with girth at least 5 satisfies $b(G) = d+1$.  For cubic graphs, if
their girth is at least $6$ or have at least $81$ vertices, then by
the above their $b$-chromatic number is $4$. Blidia, Maffray and
Zemir [12] showed that the Petersen graph provides a negative answer
to this question since they proved that the $b$-chromatic number of
Petersen graph is 3. They also proved that El Sahili and Kouider
conjecture is true for $d\leq 6$ except for Petersen graph. Cabello
and Jakovac [24] proved that a $d$-regular graph of girth at least 5
has a $b$-chromatic number at least $\lfloor\frac{d+1}{2}\rfloor$.
Besides, they proved that  every $d$-regular graph ($d\geq 6$) that
contains no cycle of order 4 and its diameter is at least $d$, has
$b$-chromatic number $d+1$. S. Shaebani [26] improved this result by
proving that $b(G)=d+1$ for
 a $d$-regular graph $G$ that contains no cycle of order 4 and
$diam(G)\geq 6$. Also, it was shown in [26]  that if $G$ is a
$d$-regular graph that contains no cycle of order 4, then $b(G)\geq
\lfloor\frac{d+3}{2}\rfloor$, and if $G$ has a triangle then
$b(G)\geq \lfloor\frac{d+4}{2}\rfloor$.\\
In this paper we'll study the $b$-chromatic number  for $d$-regular
graphs that contain no cycle of order 4 and for that of girth 5. For
this purpose, we introduce a new parameter, the $f$-chromatic vertex
number of a $d$-regular graph $G$, which is the maximum number of
dominant vertices of distinct colors in a $(d+1)$-coloring of $G$.
It is denoted by $f(G)$. First we prove that $f(G)\leq b(G)$ for
any $d$-regular graph $G$. This result allows us to establish lower
bounds of $b(G)$ by studying $f(G)$ which seems more appropriate. We
improve Shaebani result in case $d$ is even by proving that
$f(G)\geq \lceil\frac{d-1}{2}\rceil+2$ for a $d$-regular graph $G$
that contains no cycle of order 4. Also we prove that $b(G)=d+1$ for
a $d$-regular graph $G$ that contains neither a cycle of order 4 nor
a cycle of order 6, and we provide a condition on the vertices of a
$d$-regular graph $G$ that contains no cycle of order 4 in order for
$b(G)$ to be $d+1$. Finally, we show that
$f(G)\geq\lceil\frac{d-1}{2}\rceil+4$ for a $d$-regular graph $G$ of
girth 5 and diameter 5.
\begin{theorem}
Let $G$ be a $d$-regular graph, then $f(G)\leq b(G)$.
\end{theorem}
\begin{proof}
Color the vertices of $G$ by a proper $(d+1)$-coloring $c$. If
$f(G)=d+1$, then $b(G)=f(G)$. Suppose that $f(G)<d+1$, then there
exists a color class, say $i$, containing no dominant vertex. Color
each vertex $v$ of color $i$ by  $j$ such that $j \notin c(N(v))$
and so we get a proper $d$-coloring. Each dominant vertex in the
previous $(d+1)$-coloring is a dominant one in the new $d$-coloring,
so $f(G) $ is less than or equal to the number of dominant vertices
of distinct colors in the new $d$-coloring. If the number of
dominant vertices of distinct colors in the new $d$-coloring is
equal to $d$, then $ b(G)=d$ and so $f(G)\leq b(G)$. Else, we repeat
the elimination of the color classes having no dominant vertices,
one by one, until we reach a $k$-coloring with $k$  dominant
vertices of distinct colors. Thus, $k\leq b(G)$. But each dominant
vertex in the $(d+1)$-coloring is a dominant one in the
$k$-coloring, then $f(G)\leq k\leq b(G)$.
\end{proof}

\section{Regular Graphs With No Cycle of Order $4$}
\noindent Consider a $d$-regular graph $G$ and let $K$ and $F$ be
two disjoint and fixed induced subgraphs of $G$. Suppose that the
vertices of $K$ are colored by a $(d+1)$-coloring and then the
vertices of $F$ are colored by another $(d+1)$-coloring $c$ such
that the two colorings use the same set of colors. An edge of $G$ is
said to be a bad edge if it is incident with two vertices of the
same color. If the color of the two vertices is $i$, then the bad
edge is said to be an $i$-bad edge. We denote by $b_{c}$, the number
of bad edges between vertices in $F$ and others in $K$ that are
resulted from the
coloring $c$.\\
 First, we improve El Sahili and Kouider result [1]
by proving the following:
\begin{theorem}Let G be a d-regular graph, $d\geq 7$,  containing neither a cycle of
order $4$ nor of order $6$. Then, the b-chromatic number of G is $d
+ 1$.
\end{theorem}
\begin{proof}
Consider a vertex $x$ and its $d$ neighbors $x_1,x_2,..., x_d$.
Since $G$ has no cycle of order 4 then $x_i$ has at most one common
neighbor with $x$ for every $i$, $1\leq i\leq d$, and any two neighbors of $x$ don't have a common
neighbor distinct from $x$. Also, for the same reason, if $x_i$ and
$x_j$ are adjacent then a neighbor of $x_i$ is not adjacent to a
neighbor of $x_j$  for every $i$ and $j$ such that $1\leq i\neq j\leq d$. Besides, a neighbor of $x_i$ has at most one
neighbor in $\cup_{j\neq i}N(x_j)$ since $G$ has neither a cycle of
order 4 nor of order 6 for every $i$, $1\leq i\leq d$.
 Give $x$ the color $d+1$ and each vertex $x_i$ the
color $i$ for $i=1,...,d$. The vertex $x$ is then a dominant vertex. Now we will
color the neighbors of  $x_i$ in such a way that $x_i$ is a dominant
vertex for all $ i,\; 1 \leq  i \leq d$. Color the uncolored
neighbors of $x_1$  in such a way that $x_1$ is a dominant vertex.
Suppose that all the uncolored neighbors of $x_1,...,
x_{k-1},\;1\leq k-1 \leq d-1$, are colored such that $x_i$ is a
dominant vertex for  $i=1 ...
 k-1$. Let $K$ be the subgraph induced by the colored vertices other than $x_k$ and let $F$ be the
 subgraph induced by $x_k$ and its uncolored neighbors. Consider a coloring  $c$  of the uncolored
 vertices in $F$ such that $c(V(F))\subset \{1,...,d\}$, $b_c$ is minimal and $x_k$  has a neighbor in each color
class distinct from $k$. If no bad edge exists, then $x_k$  is a
dominant vertex in a proper coloring of $V (F) \cup V (K)$. Else,
suppose that there exists an $i$-bad edge and let $u$ be the end
vertex of this edge in $F$. Note that $|F\setminus\{u,x_k\}|\geq
d-3$. If $|F\setminus\{u,x_k\}|> d-3$, then since $u$ has a neighbor
of color $i$ in $K$ and at most $d-2$ vertices of color $i$ in $K$
have neighbors in $F$, there exists a vertex $y$ in
$F\setminus\{u,x_k\}$ such that $y$ has no neighbor of color $i$ in
$K$. Then switch $i$ and $c(y)$ in $F$. Otherwise; i.e.
$|F\setminus\{u,x_k\}|= d-3$, then in this case at most $d-3$
vertices of color $i$ in $K$ have neighbors in $F$. Thus  since $u$
has a neighbor of color $i$ in $K$, there exists a vertex $y$ in
$F\setminus\{u,x_k\}$ such that $y$ has no neighbor of color $i$ in
$K$. Then switch $i$ and $c(y)$ in $F$. Thus, in both cases we can
recolor the  neighbors of $x_k$, distinct from $x$, by a new coloring $c'$ such
that $x_k$ has a neighbor in each color class distinct from $k$ and
$b_{c'}< b_c$, a contradiction. Hence, the uncolored neighbors of
$x_k$ can be colored without the appearance of bad edges. Once all
the neighbors of $x_i$, for $i=1,...,d$, are colored, we complete by
giving to each other uncolored vertex a convenient color.
\end{proof}
\begin{theorem}
Let $G$ be a $d$-regular graph, $d\geq 7$, with no cycle of order 4 such that
there exists a vertex $x$ in $G$ with $N(x)=\{x_1,...,x_d\}$ that satisfies the following conditions:\newline\\
1- $|\{v\in N(x_i)$:
 $N(v)\cap (\cup_{j\neq i} N(x_j))\neq\phi\}|\leq \lceil
 \frac{d-1}{2}\rceil-1$, for every $i$, $1\leq i\leq d$,\\
2- $|N(v)\cap (\cup_{j\neq i} N(x_j))|\leq\lceil
\frac{d-1}{2}\rceil-1$ for every $v\in
  N(x_i)$ and for all $i$, $1\leq i\leq d$.\\\\
Then, $b(G)=d+1$.
\end{theorem}
\begin{proof}
Since $G$ has no cycle of order 4, we remark that  $x_i$ and $x_j$
have no common neighbor distinct from $x$, a neighbor of $x_i$ is
adjacent to at most one neighbor of $x_j$, $x_i$ has at most one
common neighbor with $x$
  and if $x_i$ is adjacent to $x_j$ then a neighbor of $x_i$ is not
  adjacent to any of the neighbors of $x_j$  for every $i$ and $j$ such that $1\leq i\neq j\leq d$. Color $x$ by $d+1$ and $x_i$ by $i$ for
$i=1,..., d$. We will color the uncolored neighbors of $x_i$ in a
way that makes $x_i$ a dominant vertex. Suppose that all the
neighbors of $x_1,...,x_{k-1}$, $1\leq k-1\leq d-1$, are colored such that
$x_i$ is a dominant vertex for every $i$, $1\leq i\leq k-1$. Let $K$ be the
subgraph induced by the colored vertices, distinct from $x_k$, and
let $F$ be the subgraph induced by $x_k$ and its uncolored
neighbors. Consider a coloring  $c$  of the uncolored
 vertices in $F$ such that $c(V(F))\subset \{1,...,d\}$, $b_c$ is minimal and $x_k$  has a neighbor  in each color
class distinct from $k$.  If no bad edge exists, then $x_k$  is a
dominant vertex in a proper coloring of $V (F) \cup V (K)$. Else,
suppose that there exists an $i$-bad edge and let $u$ be the end
vertex of this edge in $F$. Since $|N(u)\cap (\cup_{j\neq k}
N(x_j))|\leq\lceil \frac{d-1}{2}\rceil-1$, $|\{v\in N(x_k):
 N(v)\cap (\cup_{j\neq k} N(x_j))\neq \phi\}|\leq \lceil \frac{d-1}{2}\rceil-1$ and $u$ has a neighbor in $K$, then there
exists a vertex $y$ in $F\setminus\{x_k\}$ such that $y$ has no
neighbor of color $i$ in $K$ and $u$ has no neighbor of color $c(y)$
in $K$. Hence, switch $i$ and $c(y)$ in $F$. So, the $i$-bad edge is
removed and so we can recolor the  neighbors of $x_k$, distinct from $x$, by a
new coloring $c'$ such that $x_k$ has a neighbor in each color class
distinct from $k$ and $b_{c'}< b_c$, a contradiction. Thus, the
uncolored neighbors of $x_k$ can be colored without the appearance
of bad edges. Once all the neighbors of $x_i$ are colored we
complete by giving to each other uncolored vertex a convenient color.
\end{proof}\\\\
We have mentioned above that Shaebani  introduced  in [26] a lower
bound for the $b$-chromatic number of $d$-regular graphs with no
cycle of order 4. In the following theorem, we improve Shaebani
bound by 1 in the even case of $d$.
\begin{theorem}
let $G$ be a $d$-regular graph, $d\geq 7$, with no cycle of order 4
then $f(G)\geq \lceil\frac{d-1}{2}\rceil+2$.
\end{theorem}
\begin{proof}
Consider a vertex $x$ and its $d$ neighbors $x_1,x_2,..., x_d$. As
above we remark that  $x_i$ and $x_j$ have no common neighbor
distinct from $x$, a neighbor of $x_i$ is adjacent to at most one
neighbor of $x_j$, $x_i$ has at most one common neighbor with $x$
  and if $x_i$ is adjacent to $x_j$ then a neighbor of $x_i$ is not
  adjacent to any of the neighbors of $x_j$ for every $i$ and $j$ such that $1\leq i\neq j\leq d$.
   Hence the set of edges $\{e=uv: uv\in N(x)\}$ is a matching
  and so we may suppose that the vertices  $x_1,\allowbreak
x_2,...,x_{\lceil\frac{d-1}{2}\rceil+1}$ are enumerated in such a
way that at most one of these vertices has a neighbor in the set $\{x_{\lceil\frac{d-1}{2}\rceil+2},...,x_d\}$.
 Without loss of generality, if any of the vertices
in $\{x_1,x_2,...,x_{\lceil\frac{d-1}{2}\rceil+1}\}$ has a neighbor
in $\{x_{\lceil\frac{d-1}{2}\rceil+2},...,x_d\}$ then suppose that
$x_1$ is that vertex. Color $x$ by $d+1$ and $x_i$ by $i$ for
$i=1,...,d$. Now let us color the uncolored neighbors of $x_i$ in a
way that makes $x_i$ a dominant vertex for $i=1,...,\lceil
\frac{d-1}{2}\rceil+1$. Color the uncolored neighbors of $x_1$ in
such a way that $x_1$ is a dominant vertex. Suppose that all the
uncolored neighbors of $x_1$,...,$x_{k-1}$ are colored such that
$x_i$ is a dominant vertex for every $ i, \;1\leq i\leq k\leq \lceil
\frac{d-1}{2}\rceil+1$. Let $K$ be the subgraph induced by the
colored vertices, distinct from $x_k$, and let $F$ be the subgraph
induced by $x_k$ and its uncolored neighbors. Color the uncolored
neighbors of $x_k$
 by a coloring $c$ such that $c(V(F))\subset\{1,...,d\}$, $b_c$ is minimal and $x_k$  has a neighbor in each color
class distinct from $k$. If no bad edge exists, then $x_k$  is a
dominant vertex in a proper coloring of $V(F)\cup V(K)$. Else,
suppose that there exists an $i$-bad edge and let $u$ be the end
vertex of this edge in $F$. Let $C_i$ be the set of vertices of
color $i$ in $K$ that have neighbors in $F$, $t$ be the number of
vertices in $F$, $A$ be the set of vertices in $F$, distinct from
$x_k$, that have no neighbor of color $i$ in $K$ and let $s=|A|$.
Note that $|C_i|\leq \lceil\frac{d-1}{2}\rceil$. We will study two
cases:\begin{enumerate}
\item $|C_i|\leq \lceil\frac{d-1}{2}\rceil -1$\\
If $t=d$ then $s\geq (d-1)-|C_i|\geq \lfloor\frac{d+1}{2}\rfloor$.
But $|c(N(u)\cap K)\setminus\{i\}|\leq \lceil\frac{d-1}{2}\rceil-1$,
then there exists a vertex $y$ in $A$ such that $u$ has no neighbor
of color $c(y)$ in $K$. Hence, switch $i$ and $c(y)$ in $F$.
Otherwise; i.e. $t=d-1$, then $|c(N(u)\cap K)\setminus\{i\}|\leq
\lceil\frac{d-1}{2}\rceil-2$ and $s\geq (d-2)- |C_i|\geq
\lfloor\frac{d+1}{2}\rfloor-1$. Thus, there exists a vertex $y$ in
$A$ such that $u$ has no neighbor of color $c(y)$ in $K$. Hence,
switch $i$ and $c(y)$ in $F$.
\item $|C_i|= \lceil\frac{d-1}{2}\rceil $\\
Then $k=\lceil\frac{d-1}{2}\rceil+1$, $i$ is not a dominant color in
$K$, $x_k$ has no common neighbor with $x$ and $x_j$ has no common
neighbor with $x$ of color $i$ for $ j=1,...,
\lceil\frac{d-1}{2}\rceil$. In this case, $s\geq
(d-1)-\lceil\frac{d-1}{2}\rceil \geq \lfloor\frac{d+1}{2}\rfloor-1$.
If $u$ has a neighbor of color $k$ in $K$ then $|c(N(u)\cap
K)\setminus\{i,k\}|\leq \lceil\frac{d-1}{2}\rceil-2$ and so there
exists a vertex $y$ in $A$ such that $u$ has no neighbor of color
$c(y)$ in $K$. Hence, switch $i$ and $c(y)$ in $F$. Otherwise; i.e
$u$ has no neighbor of color $k$ in $K$, then switch $i$ and $k$ in
$F$.
\end{enumerate}
Thus, in both cases we can recolor the  neighbors of $x_k$, distinct from $x$,
by a new coloring, say $c'$, such that $x_k$ has a neighbor in each
color class distinct from $k$  and $b_{c'}< b_c$, a contradiction.
Hence, the uncolored neighbors of $x_k$ can be colored without the
appearance of bad edges. Once all the neighbors of $x_i$ for
$i=1,...,\lceil\frac{d-1}{2}\rceil+1$ are colored, recolor $x_{
\lceil\frac{d-1}{2}\rceil+2},...,x_d$ in such a way that makes $x$ a
dominant vertex of color $d+1$ and  keeps $x_1$ a dominant vertex
and then complete by giving to each other uncolored vertex a convenient color.
\end{proof}
\section{Regular Graphs of Girth 5 and Diameter 5}
In what follows, we will establish a lower bound for $b(G)$, where
$G$ is a $d$-regular graph with girth 5 and diameter 5, by giving the following lower bound for $f(G)$:
\begin{theorem} Let $G$ be a
$d$-regular graph, $d\geq 7$, with $g(G)=5$ and $diam(G)= 5$. Then $f(G)\geq
\lceil \frac{d-1}{2}\rceil+4$.
\end{theorem}
\begin{lemma}
Consider a $d$-regular graph $G$ with $g(G)=5$. Then the vertices of
$G$ can be  colored by a $(d+1)$-coloring in such a way that there
exists a vertex $x$ of $G$  such that $x$ and
$\lceil\frac{d-1}{2}\rceil+1$ neighbors of $x$ are dominant
vertices of distinct colors.
\end{lemma}
\begin{proof}
Let $x$ be a vertex of $G$ with $N(x)=\{x_1,..,x_d\}$. We remark
that $x_i$ and $x_j$ have no common neighbor distinct from $x$, a
neighbor of $x_i$ is adjacent to at most one neighbor of $x_j$ and
$x_i$ has no common neighbor with $x$ since $G$ is of girth 5 for every $i$ and $j$ such that $1\leq i\neq j\leq d$. Thus,
the proof is followed from the proof of Theorem 2.3.
\end{proof}\\
\\Consider a $d$-regular graph $G$, $d\geq 7$, with $g(G)=5$. Let $x$
be a vertex of $G$ and let $N(x)=\{x_1,x_2,...,x_d\}$. Let $U\subset N(x)$ such that
$|U|=\lceil\frac{d-1}{2}\rceil+1$. We define a
\emph{$U_x$-coloring} to be a $(d+1)$-coloring of $N(U)\cup U$ in
such a way that the vertices of $U$ are dominant ones of distinct
colors. By the previous Lemma, the $U_x$-coloring is possible for every $U\subset N(x)$ such that
$|U|=\lceil\frac{d-1}{2}\rceil+1$. Let $y$ be a vertex of $G$ such that $dist(x,y)=5$ and let $N(y)=\{y_1,y_2,...,y_d\}$.  We
remark that $x_i$ and $y_j$ have no common neighbor since $dist(x,y)=5$ for every $i$ and $j$, $1\leq i,j\leq d$. A neighbor
of $x_i$ (resp. $y_j$) is adjacent to at most one neighbor of $y_j$
(resp. $x_i$) for every $i$ and $j$, $1 \leq i, j \leq d$,  since $G$
contains no cycle of order 4. Consider a $U_x$-coloring and  let $K$ be the subgraph
induced by the colored vertices. We will denote by $D_K$ the set of dominant colors in $K$, by
$N_{K}$ the set of non dominant colors in $K$ that are distinct from $c(x)$ and by
$C_j(i)$ the set of vertices of color $i$ in $K$ that have neighbors
in $N(y_j)$, $1\leq i,j\leq d$. Note that if $i\in D_{K}$ then
$|C_j(i)|\leq\lceil
 \frac{d-1}{2}\rceil$ and if $i\in N_{K}$  then
 $|C_j(i)|\leq\lceil\frac{d-1}{2}\rceil+1$. Color  $y$  by $\beta$ and $y_j$ by
$\alpha$, where $\alpha$ and $\beta$ are in $N_{K}$, for some $j$,
$1\leq j\leq d$, and color the uncolored vertices in $N(y_j)$
using the same colors used to color the vertices in $K$ in
such a way that $y_j$ has a neighbor in each color class distinct
from $\alpha$. Let $F$ be the subgraph induced by $y_j$ and its
 neighbors. If  an $i$-bad edge appears between a vertex in $N(y_j)$ and another in $K$ for some $i$, $1\leq i\leq d$, then let $u_j(i)$ be  its end vertex
in $N(y_j)$, $A_j(i)$ be the subset of $(N(y_j)\setminus\{y\})$ such
that $c(A_j(i))\cap c(N(u_j(i))\cap K)=\phi$ and let
$B_j(i)=N(y_j)\setminus\{A_j(i),y\}$. Remark that $|c(N(u_j(i))\cap
K)|\leq \lceil \frac{d-1}{2}\rceil+1$, so $|A_j(i)|\geq
\lfloor\frac{d+1}{2}\rfloor -2$. A bad edge is called a removable
bad edge if it can be eliminated by some switching of colors in such
a way that no other new bad edge appears and $c(y)$ together with $c(y_j)$ remains in
$N_K$ after switching of colors. If the $i$-bad edge is not
removable then each vertex in $A_j(i)$ has a neighbor of color $i$
in $K$, since otherwise switch $c(z)$ and $i$ in $F$, where $z$ is a
vertex in $ A_j(i)$  having no neighbor of color $i$ in $K$. Thus,
if the $i$-bad edge is not removable  and $i\in N_K$ then
$|A_j(i)|\leq \lceil\frac{d-1}{2}\rceil$ since at most
$\lceil\frac{d-1}{2}\rceil+1$ vertices of color $i$  have neighbor
in $N(y_j)$ and $u_j(i)$ has a neighbor of color $i$ in $K$, and if
$i\in D_K$ then $|A_j(i)|\leq \lceil\frac{d-1}{2}\rceil-1$.
\begin{lemma}
If the $i$-bad edge is not removable, then $i\in D_{K}$.
\end{lemma}
\begin{proof}
Suppose  that $i\in N_{K}$, then $\{\alpha,\beta\}\subseteq
c(N(u_j(i))\cap K)$
 since otherwise let $k\in \{\alpha,\beta\}\setminus c(N(u_j(i))\cap K)$ and  switch
  $k$ and $i$ in $F$ and so the $i$-bad edge is removable. Thus,  $|A_j(i)|\geq \lfloor
\frac{d+1}{2}\rfloor$.
 $d$ is even since otherwise there exists $z\in A_j(i)$ such that $z$ has no neighbor of
color $i$ in $K$ and so the $i$-bad edge is removed by switching $i$ and $c(z)$ in $F$. Hence, $|A_j(i)|= \frac{d}{2}$ and  each
vertex in $B_j(i)\setminus\{u_j(i)\}$ has no neighbor of color $i$
in $K$. Since there are only $\frac{d}{2}$ colors (including $\alpha$
and $\beta$) which are not in  $D_{K}$, then there exists $l\in
c(B_j(i))$ such that $l\in D_{K}$. Let $v$ be the vertex in $F$ of
color $l$. Since $|A_j(i)|= \frac{d}{2}$ and $u_j(i)$ has a neighbor
of color $l$ in $K$, there exists at least one vertex in $A_j(i)$,
say $z$, such that $z$ has no neighbor of color $l$ in $K$. Recolor
$u_j(i)$  by $c(z)$, $z$ by $l$ and $v$ by  $i$. Hence the $i$-bad
edge is removable, a contradiction. Thus, $i\in D_{K}$.
\end{proof}
\begin{lemma}
If $d$ is odd and the $i$-bad edge is not removable, then
$|C_j(l)|=\lceil\frac{d-1}{2}\rceil$ for every color $l\in D_{K}$
 and  each vertex $u\in N(y_j)\setminus\{y\}$ such that $c(u)\notin D_K$ has
$\lceil\frac{d-1}{2}\rceil+1$ neighbors in $K$.
\end{lemma}
\begin{proof}
By the previous Lemma,
$i\in D_{K}$ and so $|A_j(i)|\leq \lceil \frac{d-1}{2}\rceil-1$. But
$|A_j(i)|\geq \lfloor\frac{d+1}{2}\rfloor-2$, then $|A_j(i)|= \lceil
\frac{d-1}{2}\rceil-1$ and so each vertex in
$B_j(i)\setminus\{u_j(i)\}$ has no neighbor of color $i$ in $K$. Since $|c(N(u_j(i))\cap K)|\leq
\lceil\frac{d-1}{2}\rceil+1$, then $\{\alpha,\beta\}\cap
c(N(u_j(i))\cap K)=\phi$. If there exists $z\in A_j(i)$ such that $z$ has no neighbor of color $l$ in $K$ for
some $l\in  (c(B_j(i))\setminus\{i\})$ then recolor $u_j(i)$ by $c(z)$,
$z$ by $l$ and the vertex of color $l$ in $F$  by $i$ and so the $i$-bad edge is removable, a contradiction.
Thus each vertex $z$ in $A_j(i)$ has a neighbor of color $l$ in
$K$, $\forall\; l\in (c(B_j(i))\setminus\{i\})$.
$l\in D_{K}$ $\forall\; l\in c(B_j(i))$, since otherwise if $l\notin
D_{K}$ for some $l\in(c( B_j(i))\setminus\{i\})$ then recolor $y_j$ by $l$, $i$ by $\alpha$ and the vertex of
color $l$ in $F$ by $i$ and so the $i$-bad edge is removable. Thus
$|C_j(l)|=\lceil \frac{d-1}{2}\rceil$ $\forall\; l\in c(B_j(i))$.
Since $|B_j(i)|= \lceil \frac{d-1}{2}\rceil+1$, then
$|C_j(l)|=\lceil \frac{d-1}{2}\rceil$ for every $ l\in D_{K}$ and each
vertex $z\in N(y_j)\setminus\{y\}$ such that $c(z)\notin D_{K}$ has
$\lceil\frac{d-1}{2}\rceil+1$ neighbors in $K$.
\end{proof}
\begin{lemma}
If $d$ is even and there exists a vertex $v\in N(y_j)$ such that
$c(v)\in N_{K}$ and $v$  has at most $ \frac{d}{2}$
neighbors in $K$ and if the $i$-bad edge is not removable, then
$|C_j(l)|=\frac{d}{2}$ for every color $l\in D_{K}$ except for at
most one color in $D_K$ and  each vertex $u\in N(y_j)\setminus\{y\}$ such
that $c(u)\notin D_K$ or $u=u_j(i)$ has at least $\frac{d}{2}$ neighbors in $K$.
 \end{lemma}
 \begin{proof}
Since the $i$-bad edge is not removable, then by Lemma 3.2, $i\in D_{K}$ and so
$|A_j(i)|\leq \frac{d}{2}-1$.  Thus $|A_j(i)|$ is either equal
 $ \frac{d}{2}-1$ or $\frac{d}{2}-2$ and so $\frac{d}{2}\leq |N(u_j(i))\cap K|\leq \frac{d}{2}+1$. Assume that
$|A_j(i)|=\frac{d}{2}-1$. Then $|\{\alpha,\beta\}\cap
c(N(u_j(i))\cap K)|\leq 1$ since $|c(N(u_j(i))\cap K)|\leq
\frac{d}{2}+1$. Let $k\in \{\alpha,\beta\}\setminus(c(N(u_j(i))\cap
K)$. Note that non of the vertices of $B_j(i)\setminus\{u_j(i)\}$
has a neighbor of color $i$ in $K$. If there exists $z\in A_j(i)$
such that $z$ has no neighbor of color $l\in c(B_j(i))\setminus\{i\}$ then recolor
$u_j(i)$ by $c(z)$, $z$ by $l$ and the vertex of color $l$ in
$N(y_j)$ by $i$ and so the $i$-bad edge is removable, a
contradiction. Thus each vertex in $A_j(i)$ has a neighbor of color
$l$ in $K$ $\forall\; l\in c(B_j(i))$. Also, $l\in D_{K}$ $\forall\;
l\in c(B_j(i))$, since otherwise if $l\notin D_{K}$ for  some $l\in c(B_j(i))\setminus\{i\}$ then recolor in
$F$ the vertex  of color $k$ by $l$, $i$ by $k$ and the vertex of
color $l$ by $i$ and so the $i$-bad edge is removable. Hence,
$|C_j(l)|= \frac{d}{2}$ for all $l\in D_{K}$ except for at most one
 and each vertex $u\in N(y_j)\setminus\{y\}$ such that
$c(u)\notin D_{K}$ has at least $\frac{d}{2}$ neighbors in $K$. \\
Finally, suppose that $|A_j(i)|=\frac{d}{2}-2$, then
$\{\alpha,\beta\}\cap c(N(u_j(i))\cap K)=\phi$. Assume that non of
the vertices in $B_j(i)\setminus\{u_j(i)\}$ has a neighbor of color
$i$ in $K$. Then $l\in D_{K}$ $\forall \; l\in c(B_j(i))$ and each vertex in
$A_j(i)$ has a neighbor of color $l$ in $K$ $\forall\; l\in c(B_j(i))$, since otherwise the $i$-bad edge is removable.
 Thus $v\in A_j(i)$ and each vertex in $A_j(i)$ got
$\frac{d}{2}+1$ neighbors in $K$, a contradiction. Hence, there
exists a unique vertex in $B_j(i)\setminus\{u_j(i)\}$, say $y'$,
such that $y'$ has a neighbor of color $i$ in $K$. Since the $i$-bad
edge is not removable, then $l\in D_{K}$ and each vertex in $A_j(i)$
has a neighbor in $K$ of color $l$, $\forall\; l\in
c(B_j(i)\setminus\{y'\})$.
We will study the following two cases:\\\\
 1- $c(y')\in N_{K}$.\newline
   If $y'$ has no neighbor of color $l$ for
some $l\in c(B_j(i)\setminus\{y'\})$, then
 recolor the vertices in $K$ of color $i$ by $l$,
that of color $l$ by $\alpha$ and that of color $\alpha$ by $i$, and
then switch $\alpha$ and $c(y')$ in $F$ and so the $i$-bad edge is
removable, a contradiction. Thus
 $y'$ has a neighbor of color $l$ in $K$ $\forall l\in
 c(B_j(i)\setminus\{y'\})$. Hence $|C_j(l)|= \frac{d}{2}$
 $\forall\;l\in c(B_j(i)\setminus\{y'\})$ and each vertex $u\in N(y_j)\setminus\{y\}$ such that
$c(u)\notin D_{K}$ has at least $\frac{d}{2}$ neighbors in $K$.\\\\
2- $c(y')\in D_{K}$.\\
 Since $c(v)\in N_K$ then $v\in A_j(i)$ and since it has at most $\frac{d}{2}$ neighbors in $K$, then $v$ has no neighbor of color $c(y')$ in
 $K$. If there exists $l\in c(B_j(i)\setminus\{y'\})$ such that
 $y'$ has no neighbor of color $l$ in $K$
  then  recolor $v$ by $c(y')$, $u_j(i)$ by $c(v)$, $y'$ by $l$ and the vertex of color $l$ in $F$ by $i$, so the
  $i$-bad edge is removable, a contradiction. Thus
 $y'$ has a neighbor of color $l$ in $K$ for every $ l\in
 c(B_j(i)\setminus\{y'\})$. Hence $|C_j(l)|= \frac{d}{2}$
 for every $l\in c(B_j(i)\setminus\{y'\})$ and each vertex $u\in N(y_j)\setminus\{y\}$ such that
$c(u)\notin D_{K}$ has at least $\frac{d}{2}$ neighbors in
$K$.
\end{proof}\\\\
Decolor now $y_j$ and its neighbors. We introduce the following lemma:
\begin{lemma}
If for all $j$, $1\leq j\leq d$,
$|C_j(l)|=\lceil\frac{d-1}{2}\rceil$ for all $l\in D_K$ except for
at most one color in $D_K$ and if there exists $j$, $1\leq j\leq d$,
such that $y_j$ has a neighbor, distinct from $y$, having at most
$\lceil\frac{d-1}{2}\rceil$ neighbors in $K$, then $f(G)\geq
\lceil\frac{d-1}{2}\rceil+4$.
\end{lemma}
\begin{proof}
 Without loss of generality suppose
that $y_1$ has a neighbor,  say $v$, such that $v\neq y$ and $v$ has
at most $\lceil\frac{d-1}{2}\rceil$ neighbors in $K$. Color $y_1$
and $y$ by $\alpha$ and $\beta$ respectively, where
$\{\alpha,\beta\}\subseteq N_K$. Color the uncolored neighbors of
$y_1$ by the same colors  used to color the vertices in $K$ in such a way
that $y_1$ has a neighbor in each color class distinct from $\alpha$
and $c(v)\in N_K$. Let $F$ be the subgraph induced by $y_1$ and its
neighbors. If no bad edge appears between $N(y_1)$ and $K$ or if all
the bad edges that appear are removable, then upon removing all the
bad edges when exist, color the uncolored neighbors of $x$ and $y$ in such a way
that $x$ and $y$ are dominant vertices and then continue by giving
each uncolored vertex a convenient color. Otherwise, suppose that an
unremovable $i$-bad edge appears between $F$ and $K$. Since the
$i$-bad edge is unremovable then each vertex in $A_1(i)$ has a
neighbor of color $i$ in $K$ and, by Lemma 3.2, $i\in D_K$.
Then $|A_1(i)|\leq\lceil\frac{d-1}{2}\rceil-1$. If $d$ is odd then
$|A_1(i)|=\lceil\frac{d-1}{2}\rceil-1$ and so each vertex in
$B_1(i)\setminus\{u_1(i)\}$ has no neighbor of color $i$ in $K$ and
$\{\alpha,\beta\}\cap c(N(u_1(i))\cap K)=\phi$. $l\in D_K$ for every $l\in
c(B_1(i))$, since otherwise recolor $u_1(i)$ by $\alpha$, $y_1$ by
$l$ and the vertex of color $l$ in $N(y_1)$ by $i$. Since $c(v)\in
N_k$ then $v\in A_1(i)$ and since $v$ has at most
$\lceil\frac{d-1}{2}\rceil$ neighbors in $K$ then there exists $l\in
c(B_1(i)\setminus\{u_1(i)\})$ such that $v$ has no neighbor of color
$l$ in $K$. Hence, recolor $u_1(i)$ by $c(v)$, $v$ by $l$ and the
vertex of color $l$ in $F$ by $i$ and so the $i$-bad edge is
removable, a contradiction. Thus $d$ is even. Since $|A_1(i)|\geq
\frac{d}{2} -2$, then $|A_1(i)|$ is either
$\frac{d}{2}-1$ or $\frac{d}{2}-2$.
$\{\alpha, \beta\}\nsubseteq c(N(u_1(i))\cap K)$ since $|N(u_1(i))\cap
K|\leq \frac{d}{2}+1$. Let $k\in
\{\alpha,\beta\}\setminus c(N(u_1(i))\cap K)$. In case
$|A_1(i)|=\frac{d}{2}-1$, non of the vertices of
$B_1(i)\setminus\{u_1(i)\}$ has a neighbor of color $i$ in $K$. And
in case $|A_1(i)|=\frac{d}{2}-2$, at most one vertex
in $B_1(i)\setminus\{u_1(i)\}$ has a neighbor of color $i$ in $K$.
The color of each vertex in $B_1(i)$ having no neighbor of color $i$
is in $D_K$, since otherwise recolor $u_1(i)$ by $k$, the vertex of color $k$ by $l$ and that of color $l$ by $i$ in $F$,
and so the $i$-bad edge is removable. Each vertex  $z$ in $A_1(i)$ has a
neighbor of color $l$ in $K$, where $l\in c(B_1(i))$ such that the
vertex of color $l$ in $F$ has no neighbor of color $i$ in $K$,
since otherwise recolor $u_1(i)$ by $c(z)$, the vertex of color $l$ in
$F$ by $i$ and $z$ by $l$. Note that the $i$-bad edge is the only
unremovable bad edge  that appears, since $v$
has at most $\frac{d}{2}$ neighbors in $K$ and $c(v)\in N_K$. Since for every $j$, $1\leq j\leq d$,
$|C_j(l)|=\frac{d}{2}$ for every $l\in D_K$ except for
at most one color in $D_K$, then there exist at least two colors in
$D_{K}$, say $r$ and $t$, such that each vertex of color $r$ (resp.
$t$) in $N_2^K(x)$ has $d-1$ neighbors in $N^G_2(y)$. Thus, these
vertices got their $d$ neighbors with one in $K$ and the others in
$N^G_2(y)$. If $i\notin \{r,t\}$, then without loss of generality
suppose that $r\in c(N(u)\cap K)$ and the vertex of color $r$
  in $N(y_{1})$ has no neighbor of color $i$ in $K$. We are going
  to make the color $r$ a non-dominant color in $K$. Suppose that $x_r$ is the dominant
vertex of color $r$ in $K$. Choose a
  color $\theta$ such that $\theta\in N_K\setminus\{c(y),c(y_1)\}$. Recolor $x_r$ by $\theta$ and the vertex of color $\theta$ in $N(x_r)$ by
  $r$. Thus, $r$ is no more in $D_K$. Note that no $r$-bad edge appears in $K$ since each vertex of color $r$ in $N_2^K(x)$, other
  than the new one, got its $d$ neighbors with one in $K$,
   which is a neighbor of $x$, and the others in $F$.
  Also, no $\theta$-bad edge appears in $K$ since $G$ has no cycle of order $4$. Finally,
  recolor the vertices in $K$ of color $i$
by $r$, that of color $r$ by $k$ and that of color $k$ by $i$. Thus,
the $i$-bad edge  is a removable bad edge, a contradiction. If $i \in \{r, t\}$,
follow the same previous procedures in order to make $i$ a non-dominant color in $K$ and then
switch $i$ and $k$ in $F$. Thus,
the $i$-bad edge  is a removable bad edge, a contradiction.\\Hence,
all the bad edges that appears between $F$ and $K$ are removable.
Once all bad edges are removed, color the uncolored neighbors of $y$
and $x$
  in such a way that $y$ and $x$ are dominant vertices. Then continue coloring by giving each uncolored vertex a convenient color.
  Hence, we obtain at least $\lceil\frac{d-1}{2}\rceil$+4 dominant vertices of distinct colors. Thus, $f(G)\geq
  \lceil\frac{d-1}{2}\rceil+4$.
  \end{proof}\\\\
\textbf{Proof of Theorem 3.1}. Consider two vertices $x$ and $y$ such that
$dist(x,y)=5$. Let $N(x)=\{x_1,x_2,...,x_d\}$ and
$N(y)=\{y_1,y_2,...,y_d\}$.  Consider a
 $U_x$-coloring for $U=\{x_1,...,x_{\lceil\frac{d-1}{2}\rceil+1}\}$ and let $K$ be the subgraph induced by the colored vertices. Let
$F_j$ be the subgraph induced by $y_j$ and its neighbors for $j=1,..., d$.
Suppose that $\forall\; j$, $1\leq j\leq d$,  $y_j$ has a
neighbor, say $v_j$, such that $v_j$ has at most
$\lceil\frac{d-1}{2}\rceil$ neighbors in $K$.
 Color $y_j$ by $\alpha$, for $j=1,...,d$, and $y$ by  $\beta $ where $\{\alpha,\beta\}\subseteq N_K$.
Color the uncolored vertices of $F_j$ by a coloring $c_j$ whose
colors are used to color the vertices in $K$ in such a way
that $y_j$ has a neighbor in each color class distinct from $\alpha$
and $c_j(v_j)\in N_K$. If there exists $j,\; 1\leq j\leq d$, such that
no bad edge appears between $F_j$ and $K$ or any bad edge that
appears between $F_j$ and $K$ is removable, then upon removing all
the bad edges when appear  between $F_j$ and $K$, decolor the vertices in $F_k\setminus\{y\},$ for $
k\neq j$, and then color the uncolored neighbors of $x$ and $y$ in
such a way that $x$ and $y$ are dominant vertices and continue
coloring by giving to each
 vertex a convenient color. Thus $f(G)\geq \lceil\frac{d-1}{2}\rceil +4$. Otherwise,
there exists an unremovable
 bad edge  between $F_j$ and $K$ for $j=1,...,d$, and so  by Lemma 3.3 and Lemma 3.4, $|C_j(l)|=\lceil\frac{d-1}{2}\rceil$ and
 for every $ j$, $1\leq j\leq d$, and for every
color
 $l$ in $D_{K}$ except for at most one color. Thus, by Lemma
3.5, $f(G)\geq \lceil\frac{d-1}{2}\rceil+4$.
From now on, we shall assume that we can't apply the above supposition. Hence,
 there exists $p$, $1\leq p\leq d$, such that each
neighbor of $y_p$, distinct from $y$, has
$\lceil\frac{d-1}{2}\rceil+1$ neighbors in $K$. Let $V\subset N(x)$
such that $|V|=\lceil \frac{d-1}{2}\rceil+1$ with $|V\cap U|=2$  if $d$ is even and $|V\cap U|=1$  if $d$ is odd.
Decolor the vertices in $K$ and $F_j$ for $j=1,..,d$, and then consider a $V_x$-coloring.
Let $K'$ be the subgraph induced by the colored vertices. Color
 $y_p$ and $y$ by $\alpha$ and $\beta$ respectively, where $\alpha$ and $\beta$ are not
  in $D_{K'}$ and distinct from $c(x)$.
  Color the uncolored neighbors of $y_p$ by the same colors used to color the vertices of $K'$ in such a way that
$y_p$ has a neighbor in each color class distinct from $\alpha$. If
no bad edge appears between a vertex in $N(y_p)$ and another
 in $K'$ or if all the bad edges
that appears between $N(y_p)$ and $K$ are removable, then upon removing all the bad edges when exist,
color the uncolored neighbors of $y$ and $x$ in such a way that $x$ and $y$ are dominant vertices
 and finally continue by giving to each uncolored vertex a convenient color.
 Then $f(G)\geq \lceil\frac{d-1}{2}\rceil +4$.
Otherwise, suppose that an $i$-bad edge which is not removable
appears between a vertex $u$ in $N(y_p)$ and another in $K'$. If $d$ is odd, then by Lemma 3.3, each vertex $v\in N(y_p)$ such that $c(v)\notin
D_{K'}$ have $\lceil\frac{d-1}{2}\rceil+1$ neighbors in $K'$, which means that each vertex $v\in N(y_p)$ such that $c(v)\notin
D_{K'}$ have $d$ neighbors in $K\cup K'$, a contradiction. Thus $d$ is even.
  Since each neighbor of $y_p$ can have at most $
\frac{d}{2}$ neighbors in $K'$ then by Lemma 3.4
 each vertex $v\in N(y_p)$ such that $v=u$ or $c(v)\notin
D_{K'}$ got $\frac{d}{2}$ neighbors in $K'$. Hence each vertex $v\in N(y_p)\setminus\{y\}$ such
that $v=u$ or $c(v)\notin D_{K'}$ have $d-1$ neighbors
 in $K\cup K'$.  Let $B$ be the set of these vertices, then $|B|=\frac{d}{2}-1$.
  Decolor the vertices in $K'$ and $F_p$. Let $U'\subset N(y)\setminus\{y_p\}$ such that $|U'|= \lceil \frac{d-1}{2}\rceil+1$.
 Consider a $U'_y$-coloring and let $K''$ be the subgraph
 induced by the colored vertices. Color $y_p$ by $\alpha$ such that $\alpha\notin \{c(v): v\in U'\cup \{y\}\}$.
 Color the uncolored neighbors of $y_p$ by the same colors used to color the vertices
of $K''$ in such a way that $y_p$ has a neighbor in each color class distinct from $\alpha$.
 If no bad edge appears between $K''$ and $F_p$, then
 color the uncolored neighbors of $y$ in such a way that  $y$ is a dominant
vertex and then color $x$ by a color distinct from the color of y
and from that of the dominant vertices in $F_p\cup K''$.
 Color the neighbors of $x$ in such a way that $x$ is a dominant vertex.
Finally, we complete by giving to each uncolored vertex a convenient
color. Thus, $f(G) \geq \lceil \frac{d-1}{2}\rceil+ 4$. Otherwise,
suppose that an $i$-bad edge appears between $F_p$ and $K''$ and let
$u'$ be the end vertex of this bad edge in $F_p$ and let $A_i\subset
N(y_p)\setminus\{y\}$ such that $c(A_i)\cap c(N(u')\cap K'')=\phi$. $u'\notin B$ since $u'$ has a neighbor
in $K''$. Since $u'$ has $\frac{d}{2} +1$ neighbors in
$K$, then $|N(u')\cap K''|\leq  \frac{d}{2}-2$ and so $ |A_i|\geq
\frac{d}{2} +1$. Thus, $ A_i\cap B\neq\phi$. Let $z\in A_i\cap B$ and then switch $i$ and $c(z)$ in $F_p$ and so the
$i$-bad edge is deleted without causing the appearance of another
new bad edge and without changing the color of $y_p$ and $y$. Thus any bad edge that appears between $F_p$ and $K''$ can be
deleted without causing the appearance of another new bad edge and without changing the color of $y_p$ and $y$.
 Once all bad edges are deleted  color $x$ by a color distinct from the color of y
and from that of the dominant vertices in $F_p\cup K''$.
 Color the
uncolored neighbors of $x$ and $y$ in such a way that $x$ and $y$ are dominant vertices.
Finally, we complete by giving to each uncolored vertex a convenient
color and so we obtain $\lceil \frac{d-1}{2}\rceil+ 4$
dominant vertices of distinct colors in a proper $d+1$-coloring of
$G$. $\blacksquare$
 
\end{document}